\documentclass[11pt,reqno]{amsart}
\usepackage{fullpage,times,graphicx,amssymb,amsmath,psfrag,xcolor,natbib}
\usepackage[colorlinks,citecolor=bbluegray,linkcolor=ddarkbrown,urlcolor=blue,breaklinks]{hyperref}
\usepackage{algorithmic,algorithm}

\usepackage[off]{auto-pst-pdf} 

\definecolor{ddarkbrown}{rgb}{0.5,0.2,0.05} \definecolor{bbluegray}{rgb}{0.05,0,0.5}

\newtheorem{theorem}{Theorem}[section]
\newtheorem{proposition}[theorem]{Proposition}
\newtheorem{definition}[theorem]{Definition}

\renewenvironment{proof}{\textbf{Proof.}}{\QED\bigskip}
\newtheorem{assumption}[theorem]{Assumption}

\newcommand{\BEAS}{\begin{eqnarray*}}
\newcommand{\EEAS}{\end{eqnarray*}}
\newcommand{\BEA}{\begin{eqnarray}}
\newcommand{\EEA}{\end{eqnarray}}
\newcommand{\BEQ}{\begin{equation}}
\newcommand{\EEQ}{\end{equation}}
\newcommand{\BIT}{\begin{itemize}}
\newcommand{\EIT}{\end{itemize}}
\newcommand{\BNUM}{\begin{enumerate}}
\newcommand{\ENUM}{\end{enumerate}}

\newcommand{\BA}{\begin{array}}
\newcommand{\EA}{\end{array}}

\newcommand{\ones}{\mathbf 1}

\newcommand{\reals}{{\mathbb R}}

\newcommand{\symm}{{\mbox{\bf S}}}  

\newcommand{\diag}{\mathop{\bf diag}}

\newcommand{\idm}{\mathbf{I}}

\newcommand{\Expect}{\textstyle\mathop{\bf E}}

\newcommand{\Co}{{\mathop {\bf Co}}}

\newcommand{\QED}{~~\rule[-1pt]{6pt}{6pt}}

\newcommand{\argmin}{\mathop{\rm argmin}}
\newcommand{\epi}{\mathop{\bf epi}}

\newcommand{\dom}{\mathop{\bf dom}}

\let \oldsection \section
\renewcommand{\section}{\vspace{3ex plus 1ex}\oldsection}
\begin{document}
\title{Global Convergence of Frank Wolfe on One Hidden Layer Networks}

\author{Alexandre d'Aspremont}
\address{CNRS \& D.I., UMR 8548,\vskip 0ex
\'Ecole Normale Sup\'erieure, Paris, France.}
\email{aspremon@ens.fr}

\author{Mert Pilanci}
\address{Electrical Engineering Department, \vskip 0ex
Stanford University, Stanford, CA 94305, USA.}
\email{pilanci@stanford.edu}

\date{\today}
\subjclass[2010]{}

\begin{abstract}
We derive global convergence bounds for the Frank Wolfe algorithm when training one hidden layer neural networks. When using the ReLU activation function, and under tractable preconditioning assumptions on the sample data set, the linear minimization oracle used to incrementally form the solution can be solved explicitly as a second order cone program. The classical Frank Wolfe algorithm then converges with rate $O(1/T)$ where $T$ is both the number of neurons and the number of calls to the oracle.
\end{abstract}
\maketitle

\section{Introduction}\label{s:intro}
We focus on the problem of training one hidden layer neural networks using incremental algorithms, and in particular the Frank-Wolfe method. While they are of course more toy models than effective classification tools, one hidden layer neural networks have been heavily used to study the complexity of the neural network training problem in a variety of regimes and algorithmic settings.

Incremental methods in particular are classical tools for training one hidden layer networks, starting at least with the results in \citep{Brei93} and \citep{Lee96}. The Frank-Wolfe algorithm \citep{Fran56}, also known as conditional gradients \citep{levitin1966constrained} is one of the most well known methods of this type, and is used in constrained minimization problems where projection on the feasible set is hard, but solving a linear minimization oracle (LMO) over this set is tractable. This method has a long list of applications in machine learning, with recent examples including \citep{joulin2014efficient,shah2015multi,osokin2016minding,locatello2017unified,freund2017extended,locatello2017greedy,miech2017learning}.

Several other approaches have recently been used to produce convergence results on one hidden layer training problems, including gradient descent schemes \citep{Vemp18} and discretized gradient flows~\citep{Chiz18,Chiz19}. Here, in the spirit of \citep{Beng06,Ross07,Bach17} we focus on training infinitely wide neural networks which are asymptotically convex. Following \citep{Bach17}, we use an $\ell_1$ like penalty to let the algorithm decide on the location of the neurons via the solutions of the linear minimization oracle. In this setting, each iteration of Frank Wolfe, i.e. each solution of the LMO, adds a fixed number of neurons to the network.

Our contribution is twofold. While the one hidden layer training problem \citep{Song17} and the linear minimization oracle problem~\citep{Guru09} are both intractable in general, we first show, using recent results by~\citep{Erge19}, that the LMO can be solved efficiently under overparameterization and mild preconditioning assumptions. Second, we discuss convexity properties of one hidden layer neural networks in a broader setting, showing in particular that the overparameterized problem has a convex epigraph and no duality gap. Using results derived from the Shapley-Folkman theorem, we derive non-asymptotic convergence bounds on this duality gap when converging towards the mean field limit. Overall, these results seem to further confirm recent empirical findings in e.g. \citep{Zhan16} that overparameterized networks, in the ``modern regime'' described in e.g. \citep{Belk19}, are inherently easier to train.

\section{Frank Wolfe on One Hidden Layer Networks}\label{s:fw}
Given $n$ real multivariate data samples $A\in\reals^{n \times d}$ and a label vector $y\in\reals^n$, together with activation functions $\sigma_\theta : \reals^d \rightarrow \reals$, parameterized by $\theta \in \mathcal{V}$ where $\mathcal{V}$ is a compact topological vector space. For a continuous function $h(\theta) : \mathcal{V} \rightarrow \reals$, we write $\int h(\theta) d\mu(\theta)$ the action of the Radon measure $\mu$ on the function $h$. 

As in \citep{Ross07,Bach17} we focus on the following problem 
\BEQ\BA{ll}\label{eq:cnn}
\mbox{minimize} & \displaystyle\sum_{i=1}^n \left(\int \sigma_\theta(a_i) d\mu(\theta) - y_i\right)^2\\
\mbox{subject to} &\displaystyle \gamma_1\left(\int \sigma_\theta(\cdot) d\mu(\theta)\right) \leq \delta
\EA\EEQ
in the variable $\mu$, a Radon measure on $\mathcal{V}$, with parameter $\delta > 0$. Here $\gamma_1$ is the variation norm, a natural extension of the $\ell_1$ norm to the infinite dimensional setting, which we describe in detail below.

\subsection{Variation Norm}
For a Radon measure $\mu$, we write 
\[
|\mu|(\mathcal{V}) \triangleq \sup_{\substack{h(\theta) : \mathcal{V} \rightarrow [-1,1],\\h \mbox{ continuous}}} \int h(\theta) d\mu(\theta)
\]
its total variation. When $\mu$ has a density, with $d\mu(\theta) = p(\theta)d\tau(\theta)$ then $|\mu|(\mathcal{V})$ is simply equal to the $L_1$ norm of $p$.

As in \citep{Bach17}, we now write $\mathcal{F}_1$ the space of functions $f(x): \reals^d \rightarrow \reals$ such that
\[
f(x)=\int \sigma_\theta(x) d\mu(\theta)
\]
where $\mu$ is Radon measure on $\mathcal{V}$ with finite total variation. The infimum of the total variation of $\mu$ over all representations of $f$, written
\[
\gamma_1(f) \triangleq \inf \left\{|\mu|(\mathcal{V}) : f(x) =  \int \sigma_\theta(x) d\mu(\theta)\right\}
\]
is a norm called the {\em variation norm} of $f$ (see e.g. \citep{Kurk01}, or the discussion on atomic norms in \citep{Chan12}). Note that when $f$ is decomposable on a finite number of basis functions, with
\[
f(x) = \sum_{i=1}^k \eta_i \sigma_{\theta_i}(x)
\]
we have 
\[
\mu(\theta) = \sum_{i=1}^k \eta_i \delta_{\{\theta=\theta_i\}}
\]
and the total variation of $\mu$ is simply $\|\eta\|_1$, the $\ell_1$ norm of $\eta$. In this context, we can rewrite problem~\eqref{eq:cnn} as an equivalent problem
\BEQ\BA{ll}\label{eq:cnn-F1}
\mbox{minimize} & \displaystyle\sum_{i=1}^n \left(f(a_i) - y_i\right)^2\\
\mbox{subject to} &\displaystyle \gamma_1\left(f\right) \leq \delta
\EA\EEQ
which is a convex problem in the variable $f \in \mathcal{F}_1$.

\subsection{Incremental Algorithm: Frank Wolfe}

Problem~\eqref{eq:cnn-F1} is an infinite dimensional problem, but it can be solved efficiently using the Frank Wolfe method (aka conditional gradients) provided we can solve a linear minimization oracle over a $\gamma_1$ ball. The Frank Wolfe algorithm solves problem~\eqref{eq:cnn-F1} by invoking a linear minimization oracle involving the gradient at each iteration, then takes convex combinations of iterates.

\paragraph{Gradients.}
The objective of problem~\eqref{eq:cnn-F1}, namely
\BEAS
L(f) &\triangleq & \sum_{i=1}^n \left(\int \sigma_\theta(a_i) d\mu(\theta) - y_i\right)^2 \\
& = & \sum_{i=1}^n \left(f(a_i) - y_i\right)^2
\EEAS
is a smooth convex functional, whose gradient is given by
\[
L'(f)(x)=\sum_{i=1}^n g_i \delta_{\{x=a_i\}}
\]
where
\BEQ\label{eq:grad-obj}
g_i = 2 \left(\int \sigma_\theta(a_i) d\mu(\theta) - y_i\right),\quad i=1,\ldots,n.
\EEQ
if we write
\[
f(a_i)=\int \sigma_\theta(a_i) d\mu(\theta),\quad i=1,\ldots,n,
\]
for a given Radon measure $\mu$. 

\paragraph{Linear Minimization Oracle.}
Given a gradient vector $g\in\reals^n$ as in~\eqref{eq:grad-obj}, because the input space is finite, each iteration of the Frank Wolfe algorithm seeks to solve the following linear minimization oracle
\[\BA{ll}
\mbox{minimize} 	& \displaystyle \sum_{i=1}^n g_i f(a_i)\\
\mbox{subject to} & \displaystyle \gamma_1\left(f\right) \leq \delta 
\EA\]
in the variable $f\in\mathcal{F}_1$. By definition of $\mathcal{F}_1$, this is equivalent to solving
\BEQ\BA{ll}\label{eq:lmo}\tag{LMO}
\mbox{minimize} 	& \displaystyle \sum_{i=1}^n g_i \left( \int \sigma_\theta(a_i) d\mu(\theta)\right)\\
\mbox{subject to} & \displaystyle \gamma_1\left(\int \sigma_\theta(\cdot) d\mu(\theta)\right) \leq \delta 
\EA\EEQ
in the variable $\mu$, a Radon measure on $\mathcal{V}$ and parameter $\delta >0$. We have, switching sums,
\BEAS
&&\inf_{\gamma_1\left(\int \sigma_\theta(\cdot) d\mu(\theta)\right) \leq 1}~ \sum_{i=1}^n g_i \left( \int \sigma_\theta(a_i) d\mu(\theta)\right) \\
&=&\inf_{\gamma_1\left(\int \sigma_\theta(\cdot) d\mu(\theta)\right) \leq 1}~  \left( \int \left(\sum_{i=1}^n g_i \sigma_\theta(a_i)\right) d\mu(\theta)\right) \\
&\geq &  -\max_{\theta \in \mathcal{V}} \left| \sum_{i=1}^n g_i\sigma_\theta(a_i)  \right|,
\EEAS
with equality if and only if $\mu=\mu_- - \mu_+$ where both $\mu_+$ and $\mu_-$ are nonnegative measures supported on the set of maximizers of 
\BEQ\label{eq:lmo-g}
\max_{\theta \in \mathcal{V}} \left|\sum_{i=1}^n g_i \sigma_\theta(a_i) \right|
\EEQ
with the value inside the absolute value positive for $\mu_+$ (respectively negative for $\mu_-$). This means that the key to solving~\eqref{eq:lmo} is solving problem~\eqref{eq:lmo-g}. We will discuss how to solve~\eqref{eq:lmo-g} for specific activation functions in Section~\ref{ss:lmo}. We first describe the overall structure of the Frank Wolfe algorithm for solving~\eqref{eq:cnn-F1} (hence~\eqref{eq:cnn}).

\paragraph{Frank Wolfe Algorithm.} Given a linear minimization oracle, the Frank Wolfe algorithm (aka conditional gradient method, or Fedorov's algorithm) is then detailed as Algorithm~\ref{alg:fw} and, calling $L^*$ the optimum value of problem~\eqref{eq:cnn-F1}, we have the following convergence bound.

\begin{algorithm}
  \caption{Frank-Wolfe Algorithm}
  \label{alg:fw}
  \begin{algorithmic}[1]
    \REQUIRE A target precision $\varepsilon>0$
    \STATE Set $t := 1$, $\mu_1(\theta)=0$.
    \REPEAT
      \STATE Get $\mu_{d}(\theta)$ solving~\eqref{eq:lmo} for 
      \[
      g_i = 2 \left(\int \sigma_\theta(a_i) d\mu_t(\theta) - y_i\right),\quad i=1,\ldots,n,
      \]
      \STATE Set $\mu_{t+1}(\theta) := (1-\lambda_t) \mu_{t}(\theta) + \lambda_t \mu_{d}(\theta)$,\vskip 0ex for $\lambda=2/(t+1)$
    \STATE Set $t:=t+1$      
    \UNTIL{$\mathrm{gap}_t \leq \varepsilon$}
    \ENSURE $\mu(\theta)_{t_{max}}$
  \end{algorithmic}
\end{algorithm}


\begin{proposition}\label{prop:complexity}
After $T$ iterations of Algorithm~\ref{alg:fw} we have
\BEQ\label{eq:comp-bnd}
L\left(\int \sigma_\theta(\cdot) d\mu_T(\theta)\right)-L^\star \leq \frac{4R^2\delta^2}{T+1}
\EEQ
where $R^2=\sup_{\theta \in \mathcal{V}}\left\{ \sum_{i=1}^n \sigma_\theta(a_i)^2\right\}$.
\end{proposition}
\begin{proof}
The objective function is 2 smooth and the result directly follows from e.g. \citep{Jagg13} or \citep[\S2.5]{Bach17}.
\end{proof}

By construction, Algorithm~\ref{alg:fw} is designed to add a constant number of atoms to the measure $\mu(\theta)$ at each iteration. After $T$ iterations, where the method reaches a precision measured by the bound~\eqref{eq:comp-bnd}, the solution $f$ thus has $O(T)$ neurons.

\paragraph{Duality Gap.} One of the key benefits of the Frank Wolfe algorithm is that, invoking convexity of the objective, it outputs an upper bound on the duality gap as a byproduct of the linear minimization oracle~\citep{Jagg13}, computed as
\BEQ
\mathrm{gap}_t = \sum_{i=1}^n g_i \left( \int \sigma_\theta(a_i) d\mu_t(\theta) - \int \sigma_\theta(a_i) d\mu_d(\theta)\right)
\EEQ
where $\mu_t(\theta)$ is the current iterate in Algorithm~\ref{alg:fw}, and $\mu_d(\theta)$ the solution of the linear minimization oracle.

\subsection{Solving the LMO for ReLU Activation Functions}\label{ss:lmo}
The key to making Algorithm~\ref{alg:fw} tractable is efficiently solving the~\eqref{eq:lmo} problem. 
When the activation function is the ReLU, given by $\sigma_\theta(a_i)=(\theta^Ta_i)_+$, and $\mathcal{V}$ is the  Euclidean unit ball, we have
\BEAS
&& \sup_{\theta \in \mathcal{V}} \left| \sum_{i=1}^n g_i\sigma_\theta(a_i)  \right| \\
&=& \left\{ \max_{\|\theta\|_2\leq 1}g^T(A\theta)_+, \max_{\|\theta\|_2\leq 1}-g^T(A\theta)_+ \right\}
\EEAS
Under certain conditions on the data set $A$, this last maximization problem is tractable as a second order program.

\subsubsection{Spike-free matrices}
\citep{Erge19} define {\em spike-free} matrices as follows.
\begin{definition}\label{def:spkf}
A matrix $A\in\reals^{n \times d}$ is spike-free if and only if
\BEQ\label{eq:spkf}
\{(Au)_+ : u\in\reals^d , \|u\|_2 \leq 1\} = A\mathcal{B}_2 \cap \reals_+^n.
\EEQ
where $\mathcal{B}_2$ is the Euclidean unit ball.
\end{definition}

The set on the left in Definition~\ref{def:spkf} is precisely the set over which we minimize in~\eqref{eq:lmo}, while the set on the right is convex. For spike free matrices, the~\eqref{eq:lmo} is thus a convex minimization problem, with
\BEQ\label{eq:lmo-relax}
\max_{\|\theta\|_2\leq 1} \pm g^T(A\theta)_+ = \max_{\substack{\|\theta\|_2\leq 1,\\A\theta \geq 0}} \pm g^TA\theta
\EEQ
where the problem on the right is a (convex) second order cone program~\citep{Boyd03}. \citep[Lem.\,2.4]{Erge19} shows for example that whitened matrices $A\in\reals^{n \times d}$ with $n \leq d$, for which $\sigma_\mathrm{min}(A)=\sigma_\mathrm{max}(A)=1$, are spike free.

In practice then, if we let $\theta_+,\theta_-$ be the optimal solutions to the right hand side of~\eqref{eq:lmo-relax}, the corresponding optimal measures read
\[
\mu(\theta) = \lambda \delta_{\theta_-}(\theta) - (1-\lambda) \delta_{\theta_+}(\theta)
\]
where $0\leq \lambda \leq 1$.

\subsubsection{Certifying Spike-Free Matrices}\label{ss:cert-spkf}
The~\eqref{eq:lmo} problem is tractable when the matrix $A$ is spike-free. \citep[Lem.\,2.3]{Erge19} shows that this is the case when $A$ has full row rank, $n \leq d$ and 
\BEQ\label{eq:spkf-cond}
\max_{\|u\|_2\leq1} \|A^\dag (Au)_+\|_2 \leq 1.
\EEQ
We can relax the left-hand side as follows 
\BEAS
&& \max_{\|u\|_2\leq1} \|A^\dag (Au)_+\|_2\\
 & \leq & \max_{z\in[0,1]^n} \max_{\|u\|_2\leq 1} \|A^\dag \diag(z) Au\|_2\\
& = & \max_{z\in[0,1]^n} \|A^\dag \diag(z) A\|_2
\EEAS
by convexity of the norm. Now, checking
\[
\max_{z\in[0,1]^n} \|A^\dag \diag(z) A\|_2 \leq 1
\]
is equivalent to deciding whether 
\BEQ\label{eq:mat-cube}
\begin{pmatrix}
\idm & A^\dag \diag(z) A\\
A^T \diag(z) A^{\dag T} & \idm
\end{pmatrix}
\succeq 0, \mbox{$\forall z\in[0,1]^n$},
\EEQ
which is a {\em matrix cube} problem, and admits a semidefinite relaxation \citep[Prop.\,4.4.5]{Bent01} which we detail in the following proposition.

\begin{proposition}\label{prop:cube-sdp}
Suppose $A\in\reals^{n \times d}$ has full row rank and $n \leq d$. Let us call $M_A(z)\in\symm_{2d}$ the matrix in~\eqref{eq:mat-cube} and assume the following linear matrix inequality 
\BEQ\label{eq:cube-lmi}
\left\{\BA{l}
X_i \succeq \frac{\rho}{2} M_A(e_i),~ X_i \succeq - \frac{\rho}{2} M_A(e_i), \quad i=1,\ldots,n\\
\\
\sum_{i=1}^n X_i \leq \frac{1}{2} M_A(\ones) + \idm
\EA\right.
\EEQ
in the variables $X_i \in \symm_{2d}$ is feasible for $\rho=1$, where $e_i$ is the Euclidean basis, then both condition~\eqref{eq:mat-cube} and a fortiori~\eqref{eq:spkf-cond} holds and $A$ is spike free.
\end{proposition}
\begin{proof}
See \citep[Prop.\,4.4.5]{Bent01}.
\end{proof}

The semidefinite relaxation in~\eqref{eq:cube-lmi} for checking condition~\eqref{eq:mat-cube} has a constant approximation ratio equal to $\pi/2$, as we recall below. 

\begin{proposition}
If the linear matrix inequality in~\eqref{eq:cube-lmi} is infeasible for $\rho=2/\pi$, then 
\[
\begin{pmatrix}
\idm & A^\dag \diag(z) A\\
A^T \diag(z) A^{\dag T} & \idm
\end{pmatrix}
\not\succeq 0,
\]
for some $z\in [0,1]^n$.
\end{proposition}
\begin{proof}
The matrices $M_A(e_i)$ all have rank at most two, hence the approximation ratio in~\citep[Th.\,4.4.1]{Bent01} is equal to $\pi/2$.
\end{proof}

\section{Stochastic Frank Wolfe}\label{s:stoch-FW}
When the number of samples $n$ is larger than the dimension $d$ , the conditions that guarantee tightness of the SOCP for solving the linear minimization oracle in Section~\ref{ss:lmo} cannot hold.


We recall in Algorithm~\ref{alg:stoch-fw} the stochastic Frank Wolfe algorithm for minimizing objectives that are finite sums, i.e.
\[\BA{ll}
\mbox{minimize} & f(x) = \sum_{i=1}^n f_i(x)\\
\mbox{subject to} & x \in\mathcal{C}
\EA\]
in the variable $x\in\reals^d$, discussed in e.g. \citep{Haza16}. This algorithm admits the following convergence bound
\[
\Expect[ f(w_t) - f(w_*) ] \leq \frac{4LD^2}{t+2}
\]
where $L$ is the Lipschitz constant of $\nabla f$ and $D$ is the diameter of the feasible set $\mathcal{C}$, provided the size $m_t$ of the minibatch is set to 
\[
m_t = \left(\frac{G(t+1)}{LD}\right)^2
\]
where $G$ is a upper bound on the Lipschitz constant of the gradients $\nabla f_i$.

\begin{algorithm}
  \caption{Stochastic Frank-Wolfe (SFW)}
  \label{alg:stoch-fw}
  \begin{algorithmic}[1]
    \REQUIRE A target precision $\varepsilon>0$, objective function $f=\sum_{i=1}^n f_i/n$, feasible set $\mathcal{C}$ and parameters $m_t$. 
    \STATE Set $t := 1$.
    \REPEAT
            \STATE Estimate the stochastic gradient 
            \[
            \tilde \nabla f = \frac{1}{|I|}\sum_{I} f_i(x_t)
            \]
            for $I$ an $i.i.d.$ sample of indices in $[1,n]$ of size $m_t$.
            \STATE Solve the linear minimization oracle
            \[
            w_d := \argmin_{w\in\mathcal{C}}  \tilde \nabla f^\top w
            \]
            \STATE Take step $w_{t+1} := (1-\lambda_t) w_{t} + \lambda_t w_{d}$,\vskip 0ex for $\lambda=2/(t+1)$
    \STATE Set $t:=t+1$      
    \UNTIL{$t \geq t_\mathrm{max}$}
  \end{algorithmic}
\end{algorithm}

Focusing on problem~\eqref{eq:cnn}, when solving problems where $n$ is larger than $d$, i.e. problems that are not overparameterized, Algorithm~\ref{alg:stoch-fw} solves a linear minimization oracle at each iteration on a {\em subset of the samples} that we write $A_I\in\reals^{m_t \times d}$. For small values of $m_t$, this matrix is much more likely to satisfy the spike-free condition in~\eqref{eq:spkf}.

Let us define $m_A$ as the largest value for which $A_I\in\reals^{m_t \times d}$ satisfies the spike-free condition in~\eqref{eq:spkf} for all subsets $I\subset [1,n]$ with $|I|\leq m_A$. To ensure that the LMO is always tractable we can limit the number of iterations so that
\[
m_t=\left(\frac{G(t+1)}{LD}\right)^2\leq m_A
\]
or again $t_\mathrm{max} \leq {LD\sqrt{m_A}}/{G}-1$. The stochastic Frank Wolfe Algorithm~\ref{alg:stoch-fw} will then solve~\eqref{eq:cnn} and yield an iterate
\BEQ\label{eq:sfw-bnd}
\Expect[ f(w_{t_\mathrm{max}}) - f(w_*) ] \leq \frac{4LD^2}{\frac{LD\sqrt{d}}{G}+1} = O\left(\frac{GD}{\sqrt{m_A}}\right).
\EEQ
which is the precision limit imposed on the algorithm by the spike-free properties of the matrix $A$. In other words, depending on the spike free properties of $A$ measured by $m_A$, the stochastic Frank Wolfe algorithm will be guaranteed to reach a precision at least equal to the bound in~\eqref{eq:sfw-bnd}.

\section{Hidden Convexity}\label{s:cvx}
The results of the previous section highlight the fact that solving problem~\eqref{eq:cnn} becomes easier as the network becomes increasingly overparameterized. The Frank Wolfe algorithm adds a couple of neurons per iteration and we have seen above that the linear minimization oracle becomes easier when $d$ is relatively large. This phenomenon, akin to the hidden convexity of the S-Lemma~\citep[\S 4.10.5]{Bent01} for example, has been observed empirically many settings, and has several geometrical roots which we discuss below. 

\subsection{Large Dimensional Regime}
Of course, as in \citep[Lem.\,2.4]{Erge19}, if $d$ is large enough, the number of neurons $m$ is larger than $n$ and the vectors $\theta_k$ are picked in general position, then the matrix with columns $(A\theta_k)_+$ for $k=1,\ldots,m$ has full rank, and we can solve~\eqref{eq:cnn} by solving a simple linear system. This of course offers no guarantee that an algorithm such as Frank Wolfe or the stochastic gradient method will converge since the problem is still nonconvex, even though the results of Section~\ref{s:fw} show that Frank Wolfe does indeed converge in this scenario, but it shows that this overparameterized regime it is inherently easier. 

\subsection{Large Number of Neurons}
Perhaps more surprisingly, a similar phenomenon occurs when the number of neurons gets larger relative to the number of samples, and the training problem becomes increasingly close to being convex. In the large number of heterogeneous neurons regime, the Shapley-Folkman theorem, a classical result from convex analysis, shows that the Minkowski sum of arbitrary sets of about the same size becomes arbitrarily close to its convex hull as the number of sets grows while the dimension remains fixed. We briefly recall this result and its consequences in optimization in what follows.

\subsubsection{The Shapley-Folkman Theorem}
Given functions $f_i$, a vector $b \in \reals^m$, and vector-valued functions $g_i$, $i\in[n]$ that take values in $\reals^m$, we consider the following separable optimization problem
\BEQ\label{eq:p-ncvx-pb-const}\tag{P}
\BA{rll}
\mathrm{h}_{P}(u) := & \mbox{minimize} & \sum_{i=1}^{n} f_i(x_i)\\
& \mbox{subject to} & \sum_{i = 1}^n g_i(x_i) \leq b + u
\EA\EEQ
in the variables $x_i\in\reals^{d_i}$, with perturbation parameter $u\in\reals^m$. We first recall some basic results about conjugate functions and convex envelopes. 

\paragraph{Biconjugate and convex envelope.}
Given a function $f$, not identically $+\infty$, minorized by an affine function, we write
\[
f^*(y)\triangleq \inf_{x\in\dom f} \{y^{\top}x - f(x)\}
\]
the conjugate of $f$, and $f^{**}(y)$ its biconjugate. The biconjugate of $f$ (aka the convex envelope of $f$) is the pointwise supremum of all affine functions majorized by $f$ (see e.g. \citep[Th.\,12.1]{Rock70} or \citep[Th.\,X.1.3.5]{Hiri96}), a corollary then shows that $\epi(f^{**})=\overline{\Co(\epi(f))}$. For simplicity, we write $S^{**}=\overline{\Co(S)}$ for any set $S$ in what follows. We will make the following technical assumptions on the functions $f_i$ and $g_i$ in our problem.
\begin{assumption}\label{as:fi}
The functions $f_i: \reals^{d_i} \rightarrow \reals$ are proper, 1-coercive, lower semicontinuous and there exists an affine function minorizing them.
\end{assumption}
Note that coercivity trivially holds if $\dom(f_i)$ is compact (since $f$ can be set to $+\infty$ outside w.l.o.g.). When Assumption~\ref{as:fi} holds, $\epi(f^{**})$, $f_i^{**}$ and hence $\sum_{i=1}^{n} f_i^{**}(x_i)$ are closed \citep[Lem.\,X.1.5.3]{Hiri96}. Also, as in e.g. \citep{Ekel99}, we define the lack of convexity of a function as follows.

\begin{definition}\label{def:rho}
Let $f: \reals^{d} \rightarrow \reals$, we let 
\BEQ\label{eq:rho}
\rho(f)\triangleq \sup_{x\in \dom(f)} \{f(x) - f^{**}(x)\}
\EEQ
\end{definition}

Many other quantities measure lack of convexity (see e.g. \citep{Aubi76,Bert14} for further examples). In particular, the nonconvexity measure $\rho(f)$ can be rewritten as
\BEQ\label{def:alt-lack-cvx}
\rho(f)=\sup_{\substack{x_i\in \dom(f)\\ \mu\in\reals_+^{d+1},\ones^\top\mu=1}}~\left\{ f\left(\sum_{i=1}^{d+1}\mu_i x_i\right) -  \sum_{i=1}^{d+1}\mu_i f(x_i)\right\}
\EEQ
when $f$ satisfies Assumption~\ref{as:fi} (see \citep[Th.\,X.1.5.4]{Hiri96}).

\paragraph{Bounds on the duality gap.}
Let $\mathrm{h}_{P}(u)^{**}$ be the biconjugate of $\mathrm{h}_{P}(u)$ defined in~\eqref{eq:p-ncvx-pb-const}, then $\mathrm{h}_{P}(0)^{**}$ is the optimal value of the dual to~\eqref{eq:p-ncvx-pb-const} (this is the perturbation view on duality, see \citep[Chap.\,III]{Ekel99} for more details). Then, \citep[Lem.\,2.3]{Ekel99}, and \citep[Th.\,I.3]{Ekel99} show the following result.

\begin{theorem}\label{th:sf}
Suppose the functions $f_i,g_{ji}$ in problem~\eqref{eq:p-ncvx-pb-const} satisfy Assumption~\ref{as:fi} for $i=1,\ldots,n$, $j=1,\ldots,m$. Let
\BEQ\label{eq:sf-pbar}
\bar p_j = (m+1) \max_i \rho(g_{ji}), \quad \mbox{for $j=1,\ldots,m$}
\EEQ
then 
\BEQ\label{eq:sf-bnd}
\mathrm{h}_{P}(\bar p) \leq \mathrm{h}_{P}(0)^{**} + (m+1)\max_i \rho(f_i).
\EEQ
where $\rho(\cdot)$ is defined in Def.~\ref{def:rho}.
\end{theorem}
This last result shows that the optimal value of problem~\eqref{eq:p-ncvx-pb-const} is bounded above and below by the optimal values of convex problems, with the gap between these bounds decreasing in relative scale when the number of terms $n$ increases relative to the number of constraints $m$. The proof of \citep[Th.\,I.3]{Ekel99} shows in fact a much stronger result, which is that the epigraphs of those three optimization problems are nested.

\subsubsection{Wide Neural Networks}
The results above show that separable optimization problems become increasingly convex as the number of terms increases. See \citep{dAsp17} for an application of these results to multitask problems and \citep{Zhan19} for an extension to training multi-branch neural networks. This has direct implications for generic one hidden layer neural networks, as we detail below.

Given samples $a_l \in \reals^{d}$ and labels $y_l$ for $l=1,\ldots n$, consider the following unregularized (unconstrained) one hidden layer network training problem
\BEQ\label{eq:1L-prob}
\min ~ \sum_{l=1}^n \left( \sum_{i=1}^p \theta_{i0}\,\sigma\left( \theta_{i}^\top a_{l}\right) - y_l\right)^2
\EEQ
in the variables $(\theta_{i0},\theta_i) \in \reals^{d+1}$ for $i=1,\ldots,p$, where $\sigma(\cdot)$ is an activation function. Defining
\BEQ\label{def:g1L}
g_l(\theta) \triangleq \theta_1 \,\sigma\left(\theta_{2}^\top a_{l} \right),\quad\mbox{for $l=1\ldots,n$,}
\EEQ
in the variable $\theta=(\theta_1,\theta_2)\in\reals\times\reals^d$, the problem can be rewritten
\BEQ\label{eq:1L-const}
\BA{ll}
\mbox{minimize} & \sum_{l=1}^n \left( z_l- y_l\right)^2\\
\mbox{subject to} & \sum_{i=1}^p g(\theta_i) = z
\EA\EEQ
in the variables $\theta_i \in \reals^{d+1}$ for $i=1,\ldots,p$ and $z \in \reals^n$. Suppose we add an $\ell_\infty$ constraint on the parameters $\theta_{i}$, solving instead
\BEQ\label{eq:1L-cons-box}
\BA{ll}
\mbox{minimize} & \sum_{l=1}^n \left( z_l- y_l\right)^2\\
\mbox{subject to} & \sum_{i=1}^p g(\theta_i) = z\\  
                    & \|\theta_i\|_\infty \leq \delta, \quad i=1,\ldots,p
\EA\EEQ
in the variables $\theta_i \in \reals^{d+1}$ for $i=1,\ldots,p$ and $z \in \reals^n$. This is equivalent to
\BEQ\label{eq:1L-cons-box-pen}
\BA{ll}
\mbox{minimize} & \sum_{l=1}^n \left( z_l- y_l\right)^2 + \sum_{i=1}^p \ones_{\{\|\theta_i\|_\infty \leq \delta\}}\\
\mbox{subject to} & \sum_{i=1}^p g(\theta_i) = z
\EA\EEQ
Now, let
\[\BA{rll}
h((u,v))=&\mbox{min.} & \sum_{l=1}^n \left( z_l- y_l\right)^2 + \sum_{i=1}^p \ones_{\{\|\theta_i\|_\infty \leq \delta\}}\\
&\mbox{s.t.} & \sum_{i=1}^p g(\theta_i) \leq z + u \\
&\mbox{s.t.} & \sum_{i=1}^p g(\theta_i) \geq z - v, \quad l=1,\ldots,n
\EA\]
then Theorem~\ref{th:sf} shows
\BEQ\label{eq:sf-bnd-1layer}
\mathrm{h}_{P}(\overline{(u,v)})^{**} \leq \mathrm{h}_{P}(\overline{(u,v)}) \leq \mathrm{h}_{P}(0)^{**} 
\EEQ
with
\[
\overline{(u,v)} =  \left(\sum_{j=1,\ldots,2n+1} |\theta_{[i]0}| \right) \rho(\sigma) \ones
\]
with $|\theta_{[1]0}| \geq |\theta_{[2]0}| \geq \ldots$ and
\[
\rho(\sigma) = \max_{l=1,\ldots,n} \rho\left(\theta_1 \sigma\left( \theta_2 a_l \right)\right).
\]
for $\|\theta\|_\infty \leq \delta$. Note that $\rho(f_i)=0$ in~\eqref{eq:sf-bnd} as the objective function in~\eqref{eq:1L-cons-box-pen} is convex. This means that when $n$ remains constant and $\delta \rightarrow 0$ as the number of neurons $p\rightarrow \infty$, i.e. the trained model is not sparse or atomic (the mean field limit), then 
\[
\left(\sum_{j=1,\ldots,2n+1} |\theta_{[i]0}| \right) \rightarrow 0,
\]
hence $\overline{(u,v)} \rightarrow 0$, the bound in~\eqref{eq:sf-bnd-1layer} is asymptotically tight and problem~\eqref{eq:1L-prob} is asymptotically equal to its convex relaxation. Overall then, the bound in~\eqref{eq:sf-bnd} precisely quantifies the convergence rate of the duality gap in problem~\eqref{eq:1L-cons-box} in the mean field limit, when the number of neurons goes to infinity. Note that, when all activation functions are identical, the convergence is actually finite, but the bound also allows us to quantify convergence in the case of heterogeneous networks.

\subsection{Convex Relaxation}
Yet another take on the hidden convexity properties of problem~\eqref{eq:cnn} is given by the results in~\citep{Lema01}. Suppose we start with a problem involving a single unit 
\BEQ\label{eq:1unit}
\BA{ll}
\mbox{minimize} & \|z-y\|_2^2\\
\mbox{subject to} & \sigma(\theta^\top a_i)=z_i, \quad i=1,\ldots,n
\EA\EEQ
in the variable $\theta \in \reals^d$. If we directly form a convex relaxation for this last problem as in e.g. \citep[S2.2]{Lema01}, by taking the convex hull of its epigraph (splitting the equality into two inequality constraints), we obtain
\[\BA{ll}
\mbox{minimize} & \|z-y\|_2^2\\
\mbox{subject to} & \sum_{j=1}^{n+2} \alpha_j \sigma(\theta_j^\top a_i)=z_i, \quad i=1,\ldots,n\\
& \|\alpha\|_1\leq 2
\EA\]
in the variables $\theta_j \in \reals^d$ for $j=1,\ldots,n+2$ and $\alpha\in\reals^{d+2}$. Even though this last problem is still nonconvex, its epigraph is convex by construction and it is an explicit (geometric) convex relaxation of problem~\eqref{eq:1unit}. This last problem also happens to exactly match an unconstrained version of the original one hidden layer training problem in~\eqref{eq:cnn}. This shows once more that, in a sense, one hidden layer neural networks where the number of neurons exceeds the number of samples are just convex problems, parameterized in a nonconvex manner.

\section{Numerical Results}

\subsection{Linear Minimization Oracle}
The results discussed in Section~\ref{ss:lmo} on the linear minimization oracle guarantee that whitened matrices $A$ with $n\leq d$ are spike free, hence satisfy
\BEQ\label{eq:lmo-equiv}
\{(Au)_+ : u\in\reals^d , \|u\|_2 \leq 1\} = A\mathcal{B}_2 \cap \reals_+^n.
\EEQ
Solving the LMO under this equivalence means solving a second order cone program. To get a sense of how far this equivalence is likely to hold beyond this regime, we first check a necessary condition on whitened matrices with $n \geq d$. While the solution of the original LMO, given by 
\[
\max_{\|\theta\|_2\leq 1} g^T(A\theta)_+ 
\]
is always nonzero when $g\not\leq 0$, that of its SOCP counterpart, written
\[
\max_{\substack{\|\theta\|_2\leq 1,\\A\theta \geq 0}} \pm g^TA\theta
\]
can only be nonzero if there is a vector $\theta$ such that $A\theta\geq 0$. This means that the SOCP cannot solve the LMO if $\{\theta:A\theta\geq 0\}=\{0\}$, in other words, $\{\theta:A\theta\geq 0\}\neq \{0\}$ is a necessary condition for $A$ being spike-free and~\eqref{eq:lmo-equiv} to hold. 

We sample Gaussian matrices $A\in\reals^{n \times d}$ with $d=20$ and $n$ ranging from 20 to 75, with 200 samples at each $n$. We then whiten these matrices and check if $\{\theta:A\theta\geq 0\}\neq \{0\}$. In Figure~\ref{fig:prob-LMO}, we plot the resulting empirical probability and notice a phase transition starting a bit after $n=d$ which seems to indicate that, for Gaussian matrices at least, the overparamerization requirement is tight.

\begin{figure}[h]
    \centering
    \psfrag{n}[t][b]{$n$}
    \psfrag{Prob}[b][t]{Probability $\{\theta:A\theta\geq 0\}\neq \{0\}$}
    \includegraphics[width=0.45\textwidth]{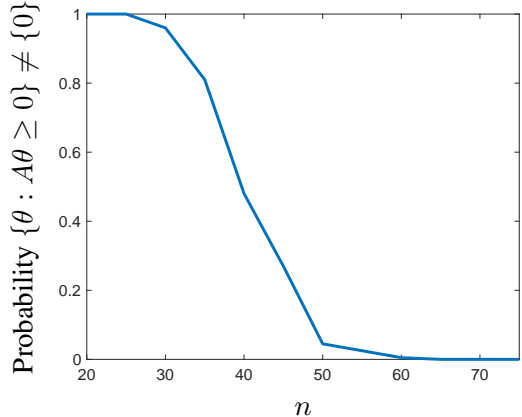}
    \caption{Probability of SOCP solving the linear minimization oracle having a nonzero solution versus number of samples $n$ for $d=20$. When the solution to the SOCP is zero, it cannot be a tight solution of the LMO. \label{fig:prob-LMO}}
\end{figure}

We also tested the matrix cube relaxation in~\eqref{eq:cube-lmi} on ten sample Gaussian matrices $A\in\reals^{n \times d}$ with $d=10$. After whitening, the linear matrix inequality in~\eqref{eq:cube-lmi} was always feasible on these samples for $n=5$ and $n=10$, showing that, in these toy examples at least, the SDP relaxation is tight enough to certify that the whitened matrices are spike-free.

On the other hand, when repeating this last experiment on Gaussian matrices that were {\em not whitened}, the linear matrix inequality in~\eqref{eq:cube-lmi} was always infeasible, showing that these matrices are potentially not spike-free. This means that some form of normalization is critical to the tractability of the linear minimization oracle.

\subsection{Frank Wolfe}
We now test the convergence of the Frank Wolfe Algorithm~\ref{alg:fw} on toy examples. In Figure~\ref{fig:conv-FW1} the ground truth is generated using a ten neurons in dimension 25 using Gaussian weights, observing 20 data points and no whitening. In Figure~\ref{fig:conv-FW10}, we repeat the same experiment using ten neurons, this time whitening the data. In Figure~\ref{fig:conv-FW20-20}, we repeat this last experiment once more at the edge of the overparameterization regime, with $d=n=20$. Convergence seems faster in the whitened examples, where the guarantees hold. 

Finally, in Figure~\ref{fig:conv-SFW} we test convergence of the Stochastic Frank Wolfe Algorithm~\ref{alg:stoch-fw} on a toy network example where the ground truth is generated by ten neurons, in dimension $d=20$ using $n=25$ samples and whitening. Note that the stochastic variant produces no valid gap. In this setting, the results in Section~\ref{s:stoch-FW} only guarantee convergence until a fixed (a priori intractable) precision threshold, which is indeed what we observe in this experiment. In cases where the spike free condition in~\eqref{eq:spkf-cond} does not hold, the SOCP typically returns a solution equal to zero (cf. Figure~\ref{fig:prob-LMO}) and convergence stalls. 

\begin{figure}[h]
    \centering
    \psfrag{Iter}[t][b]{Iterations}
    \includegraphics[width=0.45\textwidth]{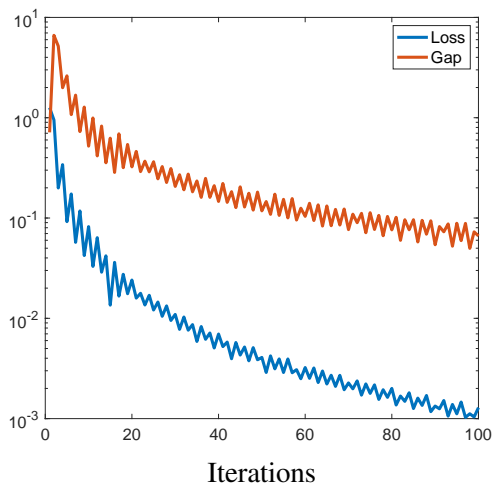}
    \caption{Convergence of Frank Wolfe on a toy network example where the ground truth is generated using ten neurons, in dimension $d=25$ using $n=20$ samples and no whitening. We plot both loss and duality gap bound versus number of iterations (and a proportional number of neurons). \label{fig:conv-FW1}}
\end{figure}

\begin{figure}[h]
    \centering
    \psfrag{Iter}[t][b]{Iterations}
    \includegraphics[width=0.45\textwidth]{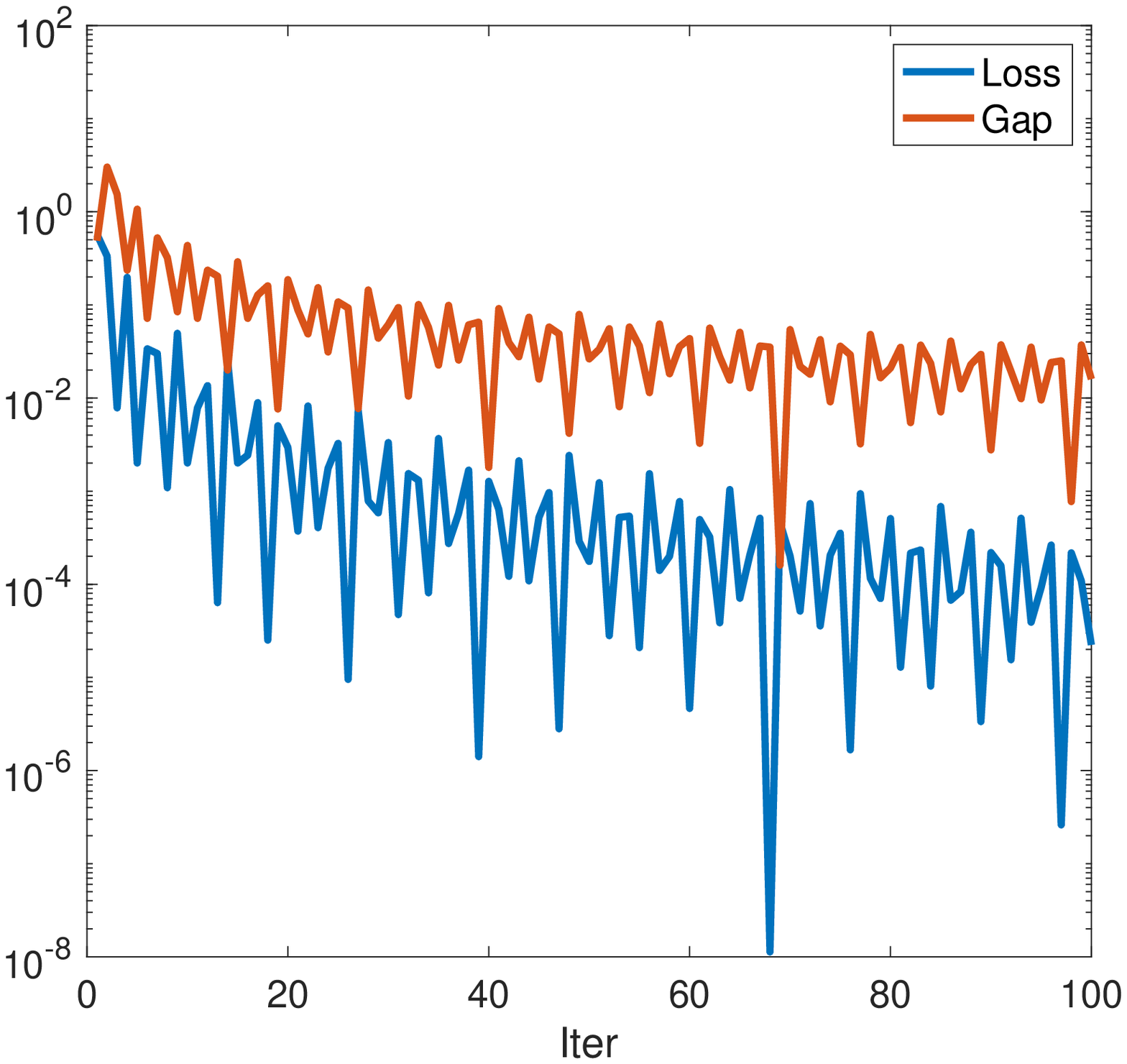}
    \caption{Convergence of Frank Wolfe on a toy network example where the ground truth is generated by ten neurons, in dimension $d=25$ using $n=20$ samples and whitening. We plot both loss and duality gap bound versus number of iterations (and a proportional number of neurons). \label{fig:conv-FW10}}
\end{figure}

\begin{figure}[h]
    \centering
    \psfrag{Iter}[t][b]{Iterations}
    \includegraphics[width=0.45\textwidth]{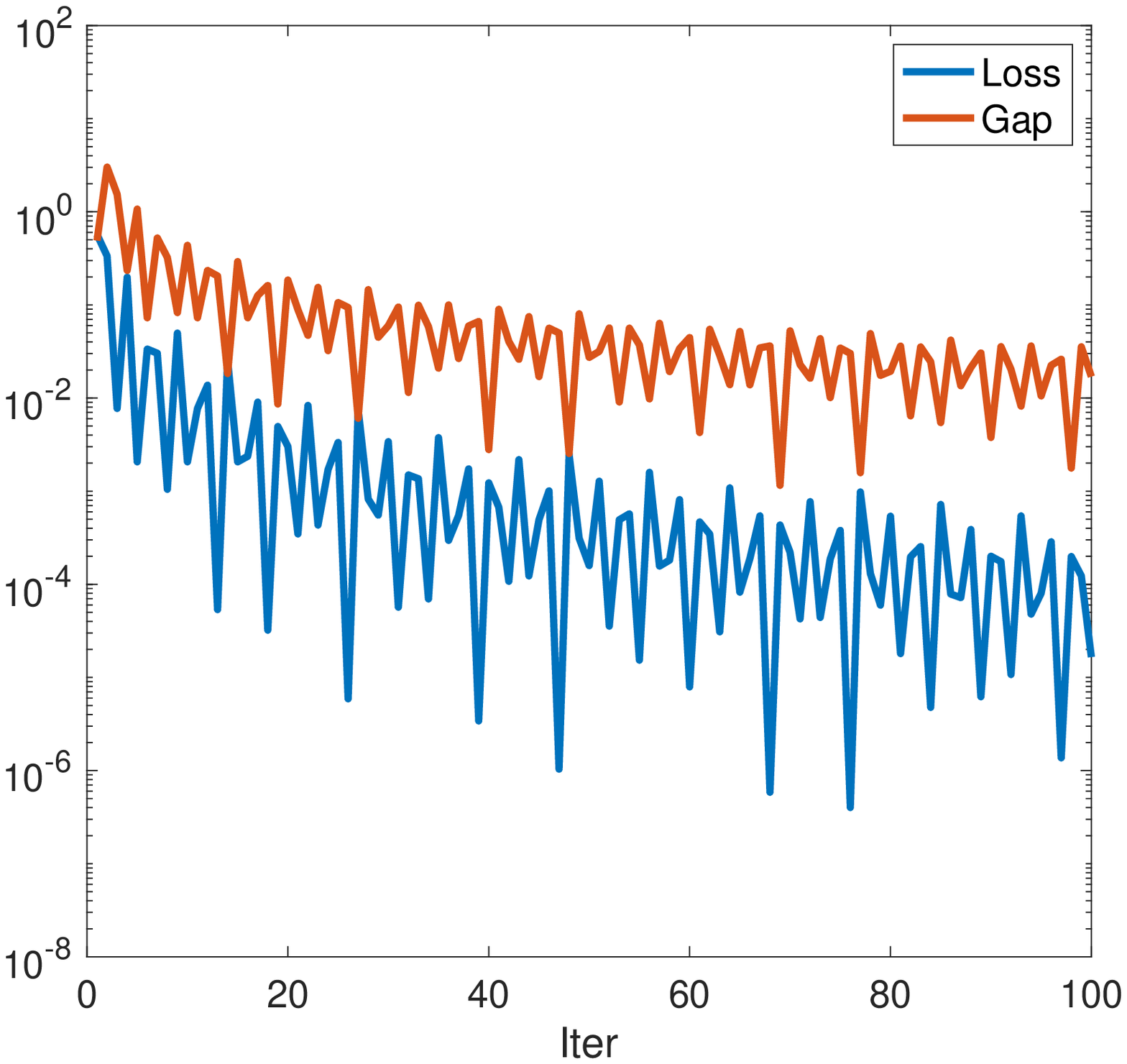}
    \caption{Convergence of Frank Wolfe on a toy network example where the ground truth is generated by ten neurons, in dimension $d=20$ using $n=20$ samples and whitening. We plot both loss and duality gap bound versus number of iterations (and a proportional number of neurons). \label{fig:conv-FW20-20}}
\end{figure}

\begin{figure}[h]
    \centering
    \psfrag{Iter}[t][b]{Iterations}
    \includegraphics[width=0.45\textwidth]{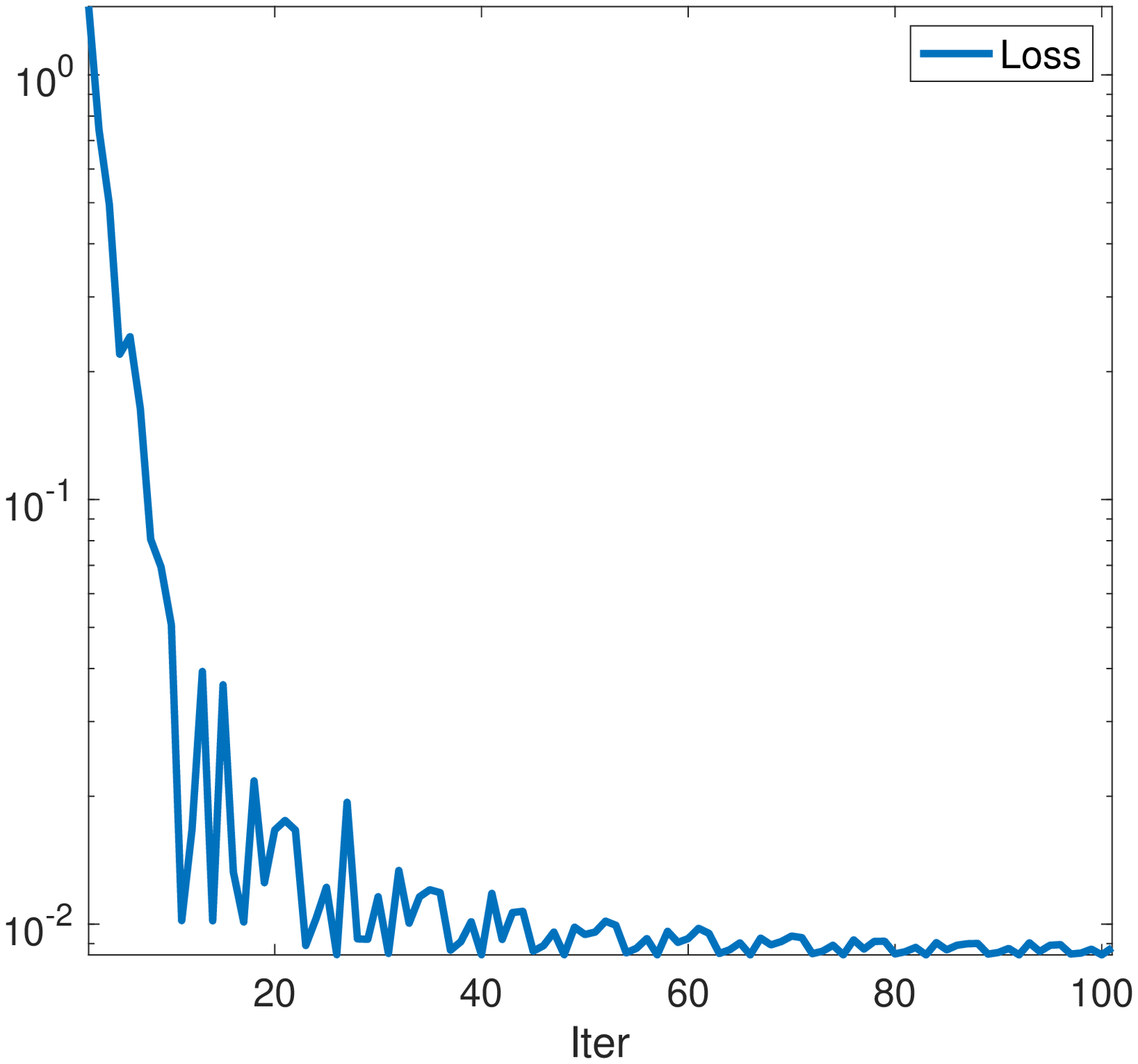}
    \caption{Convergence of the Stochastic Frank Wolfe Algorithm on a toy network example where the ground truth is generated by ten neurons, in dimension $d=20$ using $n=25$ samples and whitening. We plot loss versus number of iterations (and a proportional number of neurons). \label{fig:conv-SFW}}
\end{figure}

\section*{Acknowledgements}{AA is at CNRS \& d\'epartement d'informatique, \'Ecole normale sup\'erieure, UMR CNRS 8548, 45 rue d'Ulm 75005 Paris, France,  INRIA  and  PSL  Research  University. AA acknowledges support from the French government under management of Agence Nationale de la Recherche as part of the "Investissements d'avenir" program, reference ANR-19-P3IA-0001 (PRAIRIE 3IA Institute), the ML \& Optimisation joint research initiative with the fonds AXA pour la recherche and Kamet Ventures, as well as a Google focused award.}

\clearpage
\small{\bibliographystyle{plainnat}
\bibsep 1ex
\bibliography{/Users/aspremon/Dropbox/Research/Biblio/MainPerso.bib,fw.bib}
\end{document}